\documentclass[11pt]{amsart}
\usepackage{amssymb,amsmath,amsthm,enumerate}

\newtheorem{thm}{Theorem}[section]
\newtheorem{lem}[thm]{Lemma}
\newtheorem{cor}[thm]{Corollary}
\newtheorem{prop}[thm]{Proposition}

\newtheorem{quest}[thm]{Question}

\theoremstyle{remark}
\newtheorem{exm}{Example}
\newtheorem{rem}[thm]{Remark}
\theoremstyle{definition}
\newtheorem{defi}[thm]{Definition}

\newcommand {\Zz} {\mathbb{Z}}
\newcommand {\Nz} {\mathbb{N}}
\newcommand {\Cz} {\mathbb{C}}
\newcommand {\Qz} {\mathbb{Q}}
\newcommand {\Rz} {\mathbb{R}}
\newcommand {\Fz} {\mathbb{F}}

\newcommand {\Oc} {{\mathcal O}}

\newcommand{\BAR}{\overline}
\DeclareMathOperator{\modu}{mod \:}

\DeclareMathOperator{\Rep}{{\mathcal R}\it ep}

\newcommand{\defst}[1]{{\it #1}}

\begin{document}

\title{Fusion algebras with negative structure constants}
\author{Michael Cuntz}
\thanks{}
\address{Michael Cuntz, Universit\"at Kaiserslautern,
Postfach 3049, 67653 Kaiserslautern}
\email{cuntz@mathematik.uni-kl.de}

\begin{abstract}
We introduce fusion algebras with not necessarily positive
structure constants and without identity element. We prove that they
are semisimple when tensored with $\mathbb{C}$ and that their characters
satisfy orthogonality relations. Then we define
the proper notion of subrings and factor rings for such algebras.
For certain algebras $R$ we prove the existence of a ring $R'$
with nonnegative structure constants such that $R$ is
a factor ring of $R'$.
We give some examples of interesting factor rings of the
representation ring of the quantum double of a finite group.
Then, we investigate the algebras associated to Hadamard
matrices. For an $n\times n$-matrix the corresponding algebra
is a factor ring of a subalgebra of $\Zz[{(\Zz/2\Zz)}^{n-2}]$.
\end{abstract}

\maketitle

\section{Introduction}

A $\Zz$-based ring is a finite dimensional $\Zz$-algebra
with certain properties which make it semisimple when tensored
with $\Cz$. Actually, it is almost a fusion algebra
(or table algebra \cite{zAeFmM}, or based ring \cite{Lus1})
in the most common sense except that it may have negative structure constants
(and is commutative).
The fact that the structure constants are all nonnegative in
fusion algebras has important consequences which we miss
in the case of $\Zz$-based rings. For example, the corresponding
$s$-matrix (an anologue of a character table) will always
have a column with only nonzero positive real entries which
are sometimes called degrees. This is not the case for example
for the algebras belonging to the imprimitive ``spetses''
(see \cite{mC1}, \cite{MUG} and \cite{mBgMjM} or \cite{gMSp}
for a definition of spetses).

There are other reasons why algebras with negative structure
constants are worth to be investigated. Some of the Fourier matrices
corresponding to the exceptional ``spetses'' may be constructed using factor
rings of some fusion algebras and in general, ``factor rings'' (see \ref{facring}
for an exact definition in the case of $\Zz$-based rings) will
often have negative structure constants.

For a given $\Zz$-based ring $R$, it is an interesting task to
find a fusion algebra $R'$ such that $R$ is a factor ring $R'/I$
for some ideal $I$. In some cases, it suffices to double the basis
(add ``negatives''), but in general, it is not clear if this
is possible. We prove the existence of such a ring $R'$ for
certain $\Zz$-based rings:

\begin{thm}
Let $R$ be a $\Zz$-based rng with $s$-matrix $s$. If $s$ is of the form
 \[ s_{ij} = \mu_j v_{ij} \]
where $0\ne \mu_j\in\Zz$ and $v_{ij}$ is a root of unity or $0$, then
$R$ is a factor ring of a pointed $\Zz$-algebra with nonnegative structure
constants (an algebra with a distinguished basis).
\end{thm}

The proof is given in section \ref{idealsandfactors}.
In particular, it follows that the
ring associated to a Hadamard matrix is a factor ring of
a ring with nonnegative structure constants.

Motivated by exterior powers of character rings and by
Hadamard matrices, we generalize the definition of a $\Zz$-based
ring by dropping the assumption that the algebra should
have a one. It appears that most properties of table algebras
can be transfered to such ``$\Zz$-based rngs'' (the missing ``i''
is meant to suggest that it is a ``ring'' without an identity element).

The concept of closed subsets that has been introduced for table
algebras (see for example \cite{hiBphZ}, \cite{hB}, \cite{zAeFmM})
works in the case of $\Zz$-based rings too (\cite{mC}).
And it may also be generalized to $\Zz$-based rngs although
it is not trivial that the span of a subset
of the basis which is closed under multiplication is
again a $\Zz$-based rng. We prove (see section \ref{closedsubsets}):

\begin{thm}
If the $\Zz$-module spanned by a subset of the basis of a $\Zz$-based rng
is closed under multiplication, then it is a $\Zz$-based rng.
\end{thm}

An $n\times n$ Hadamard matrix can be interpreted as an $s$-matrix
(after an adequate rescaling), it yields a $\Zz$-algebra which is
a $\Zz$-based rng.
If $n=4k$ with $k>1$ odd, then this algebra will always have some
negative structure constants. It is a factor ring of 
a subalgebra of the group
algebra $\Zz[Z_2^{n-2}]$, where $Z_2$ denotes the cyclic group with
two elements.

In the case of Hadamard matrices, the structure constants are a
subset of the numbers used to define profiles (see \cite{wdW}, 9.3).
Profiles are useful to test if Hadamard matrices are equivalent.
The structure constants contain enough information to reconstruct
the matrix. This is a consequence of the theorem of Wedderburn-Artin
in the case of arbitrary $\Zz$-based rngs, and for Hadamard matrices
this reduces to the fact that diagonalizable commuting matrices
may be diagonalized simultaneously. Of course the resulting matrix
is unique only up to permutations of rows and columns, which does
not change the equivalence class. Multiplying rows or columns by
signs is a more complicated operation and it is less obvious when
two algebras are equivalent via sign changes, although there are good
algorithms to find this out.

The structure of the paper is as follows:

In the first section we introduce the $\Zz$-based rngs and its
$s$-matrix and prove orthogonality relations. Then we define closed
subsets for $\Zz$-based rngs and prove that for any basis element $b$,
there exists a power $b^m$ which has non zero value under the symmetrizing
trace $\tau$. This enables us to prove that closed subsets span
$\Zz$-based rngs. We also briefly describe an algorithm for finding
closed subsets.

Then we study ideals and factor rings, the objective being to realize
$\Zz$-based rngs as factor rings of fusion algebras. We give an example
of such ideals for the case of the representation ring of the quantum
double of a finite group and cite some examples of factor rings
of algebras coming from Kac-Moody algebras in the following section.
In the last section we apply our theory to the rings associated to
Hadamard matrices.

\section{$\Zz$-based rings without one}\label{zbrwo}

In \cite{mC1} the so called $\Zz$-based ring has been introduced.
This is a commutative $\Zz$-algebra $R$ given with a basis $B$ and an involution
$\sim$ with certain properties which enable us to show that
$R_\Cz:=R\otimes_\Zz \Cz$ is semisimple.
One important axiom in the definition is that the unit of the algebra
is an element of $B$. We would like to preserve the property that
$R_\Cz$ is semisimple without requiring $1\in B$.

In the following, we embed $R$ into $R_\Cz$ by
$r\mapsto r\otimes 1$ and denote $r\otimes 1$ by $r$
(this is possible because $R$ is a free $\Zz$-module).

\subsection{Definition}

Let $R$ be a finitely generated commutative $\Zz$-algebra
which is a free $\Zz$-module with basis
$B=\{b_0,\ldots,b_{n-1}\}$ and \defst{structure constants}
\[ b_i b_j = \sum_m N_{ij}^m b_m, \quad N_{ij}^m\in \Zz \]
for $0\le i,j < n$.

Assume that there is an $e\in R_\Cz$ such that
$ea=a$ for all $a\in R_\Cz$. In this case, $e$ is uniquely determined
by this condition.
%because \[ e=ee'=e' \] for elements $e,e'$ with this property.

Let $\tau_i : R \rightarrow \Cz$ be defined to be the
linear extension of $\tau_i(b_j)=\delta_{i,j}$.
If $e = \sum_i e_i b_i$ in $R_\Cz$ with $e_i\in\Cz$, then define
\[ \tau := \sum_i \bar e_i \tau_i, \]
where $\bar{}$ means complex conjugation.
Now assume the existence of an involution $\sim : R \rightarrow R$
which is a $\Zz$-module homomorphism such that
\[ \quad \tilde B= B,\quad N_{\tilde i\tilde j}^m=N_{i j}^{\tilde m},
\quad \tau(\tilde b_i b_j) = \delta_{i,j} \]
for all $0\le i,j,m < n$,
where $\tilde i$ is the index with $\tilde{b_i}=b_{\tilde i}$.

We extend $\sim$ and $\tau$ to $R_\Cz$ by
\[ \widetilde{r \otimes z} := \tilde r \otimes \bar z,
\quad \tau(r \otimes z):=z \tau(r) \]
for $r\in R$, $z \in \Cz$ and use the same symbols for the extended maps.
We require the following additional property of $\sim$:
\[ \tilde e = e. \]

\begin{defi}
We call $(R,B)$ with $e$ and $\sim$ as above a \defst{$\Zz$-based rng}.
\end{defi}

The definition of $\tau$ is motivated by $\tau(\tilde b_i b_j) = \delta_{i,j}$,
because then
\[ \tau(e)=\tau(\tilde e e)=\tau\big(\sum_{i,j} \bar e_i e_j \tilde b_i b_j\big)
=\sum_i \bar e_i e_i \]
and
\[ \tau(e)=\sum_i \bar e_i \tau_i(e) = \sum_i \bar e_i e_i. \]
Using $e=\tilde e$ we also get $\bar e_i=e_{\tilde i}$ and hence
\[ \tau(\tilde r)=\sum_j \bar e_j\tau_j(\sum_i\bar\mu_i\tilde b_i)
= \sum_j\bar e_j \bar\mu_{\tilde j}=\sum_j e_j\bar\mu_j
= \overline{\sum_j \bar e_j\tau_j(\sum_i\mu_i b_i)}
= \overline{\tau(r)} \]
for any $r=\sum_i\mu_i b_i \in R_\Cz$.

The map
\[\langle\:,\:\rangle : R\times R\rightarrow \Zz,\quad
\langle r,r' \rangle:=\tau(\tilde r r')\]
for $r,r'\in R$ behaves like an inner product with
orthonormal basis $B$ because $r = \sum_{b\in B} \langle b,r \rangle b$
for all $r\in R$. Then $\tilde B$ is the basis
dual to $B$ with respect to this inner product.

It is easily seen that
extending $\langle\:,\:\rangle$ to $R_\Cz\times R_\Cz$ yields
a hermitian positive definite sesquilinear form,
so it is non degenerate.

\begin{prop}
Let $R$ be a $\Zz$-based rng.
Then the algebra $R_\Cz=R\otimes_\Zz\Cz$ is semisimple.
\end{prop}
\begin{proof}
If $\mathfrak I$ is an ideal in $R_\Cz$, then the orthogonal complement
\[ \mathfrak I^\perp := \{ r \in R_\Cz \mid \langle r,r' \rangle = 0
 \quad \forall r' \in \mathfrak I \} \]
is a left ideal too:
\[ \langle tr,r' \rangle = \tau(\widetilde{tr} r') =
\tau(\tilde r \tilde t r') = \langle r,\tilde t r' \rangle = 0 \]
for all $r \in \mathfrak I^\perp$, $t \in R_\Cz$ and $r' \in \mathfrak I$.
The claim follows.
\end{proof}

\subsection{$s$-matrix}\label{smat}

Now $R_\Cz$ is a commutative semisimple algebra over an algebraically
closed field, so by the theorem of Wedderburn-Artin it is isomorphic
as a $\Cz$-algebra to $\Cz^n$ with componentwise multiplication.
By choosing $B$ as a basis for $R_\Cz$ and the canonical basis
$\{v_i\}_i$ with $v_i v_j=\delta_{i,j} v_i$ for all $i,j$ for $\Cz^n$, an
isomorphism $\varphi$ is described by a matrix $s$ which we will call
an \defst{$s$-matrix of} $(R,B)$:
\[ \varphi(b_i) = \sum_k s_{ki} v_k. \]
Remark that this matrix depends on the choice of the isomorphism
$\varphi$. Another isomorphism would differ from $\varphi$ by
a $\Cz$-algebra automorphism of $\Cz^n$,
so an $s$-matrix is unique up to a permutation of rows.

The algebra $R$ is the $\Zz$-lattice spanned by the columns of $s$.
Its multiplication is componentwise multiplication of vectors and
the involution $\sim$ corresponds to complex conjugation on the
columns (see lemma \ref{closedcomplex}). The structure constants are
(Verlinde's formula)
\begin{equation}\label{VF}
N_{ij}^m = \sum_l s_{li} s_{lj} s'_{ml}
\end{equation}
where $s'=s^{-1}$ (this is just the linear decomposition of the
product of two columns with respect to the columns).

\subsection{Orthogonality relations}

The rows of $s$ are the images of $B$ under the one-dimensional representations
of $R$ because $s_{ki}s_{kj}=\sum_l N_{ij}^l s_{kl}$ for all $k,i,j$.
They are orthogonal. To prove this, we need:

\begin{lem}\label{Rclin}
Let $E_1, E_2$ be $R_\Cz$-modules and $h : E_1 \rightarrow E_2$ be a $\Cz$-linear
map. Then the map $h_0 : E_1 \rightarrow E_2, w \mapsto \sum_{i \in I} b_i h(\tilde b_i w)$ is
$R_\Cz$-linear.
\end{lem}
\begin{proof}
For $j,k,m\in \{0,\ldots,n-1\}$, we have
\[\tau(\tilde b_m \sum_i N_{jk}^i b_i) = \sum_i N_{jk}^i \tau(\tilde b_m b_i) = N_{jk}^m, \]
so $\tau(b_m b_j b_k) = N_{jk}^{\tilde m}$.
Because of $\tau(\tilde r) = \overline{\tau(r)}$, we get
\[ N_{\tilde j k}^i = \tau(\tilde b_i \tilde b_j b_k)
= \overline{\tau(\tilde b_k b_j b_i)} = \overline{N_{ji}^k} = N_{ji}^k.\]
Now we can conclude that $h_0$ is a homomorphism of $R_{\Cz}$-modules:
for $r = \sum_k \mu_k \tilde b_k \in R_\Cz$ and $w\in E_1$, we have
\[ h_0(r \cdot w) = \sum_i b_i h(\tilde b_i \sum_k \mu_k \tilde b_k w) \stackrel{(*)}{=}
\sum_i b_i h(\sum_k \mu_k \sum_j N_{ki}^j \tilde b_j w) \]
\[ = \sum_k \mu_k \sum_{j,i} N_{ki}^j b_i h( \tilde b_j w) = \sum_k \mu_k \sum_{j,i} N_{\tilde k j}^i
b_i h( \tilde b_j w) \]
\[ = \sum_k \mu_k \tilde b_k \sum_{j} b_j h( \tilde b_j w) = r \cdot h_0(w), \]
where $(*)$ holds because
$ \tilde b_i \tilde b_k = \widetilde{b_k b_i} = \widetilde{\sum_j N_{ki}^j b_j}
= \sum_j N_{ki}^j \tilde b_j$.
\end{proof}

This implies:

\begin{prop}\label{orthrel}
Let $R$ be a $\Zz$-based rng.
Then its $s$-matrix has orthogonal rows.
\end{prop}
\begin{proof}
Let $E_1,E_2$ be one dimensional $R_\Cz$-modules and $\chi_1,\chi_2$
be the corresponding characters (trace of the operation of $R$ on $E_1$, $E_2$).
So $b\in B$ acts on $E_1$ respectively $E_2$ as the scalar $\chi_1(b)$
respectively $\chi_2(b)$.
For an arbitrary $\Cz$-linear map $h : E_2 \rightarrow E_1$,
lemma \ref{Rclin} states that the map $h_0$ given by
\[ h_0(w) = \sum_{b\in B} \chi_1(b) h(\chi_2(\tilde b) w) =
\Big(\sum_{b \in B} \chi_1(b) \chi_2(\tilde b)\Big) h(w) \]
is a homomorphism of $R_\Cz$-modules from $E_2$ to $E_1$.
But these modules are irreducible, so by the lemma of Schur,
$h_0$ is either an isomorphism or the zero map.
This holds for all $\Cz$-linear maps $h$, hence
\[ \sum_{b\in B} \chi_1(b) \chi_2(\tilde b) \]
is zero if $E_1\ncong E_2$, i.e. $\chi_1\ne \chi_2$.
On the other hand, if $h_0$ is an isomorphism, then the above sum can
not be zero, because there is at least one element in $E_2$
which is not mapped to zero.
\end{proof}

\begin{defi}
By normalizing the rows of the $s$-matrix we get an
orthonormal matrix called the \defst{Fourier matrix} of $(R,B)$.
\end{defi}
The Fourier matrix is unique if we choose the positive
square roots of the rows in the normalization.

\subsection{Examples}

\begin{exm}
Of course, a $\Zz$-based ring is also a $\Zz$-based rng. So
for example representation rings of certain Hopf algebras or
exterior powers of group rings of cyclic groups are $\Zz$-based rngs
(see \cite{mC1} or \cite{mC}, 5.1).
\end{exm}

\begin{exm}
Let $s=kI$ be a scalar matrix, $k\in \Zz$. Then
$b_i b_j = k \delta_{ij} b_i$.
In $R_\Cz$, $e=\sum_i \frac{1}{k} b_i$. The involution
$\sim$ is the identity map.
Further, $\tau = \sum_i \frac{1}{k} \tau_i$ and indeed,
$\tau(b_i b_i) = 1$ for all $i$.
\end{exm}

\begin{exm}\label{exkkk}
Let $s$ be of the form
\[ s = \begin{pmatrix}
k & \\
\vdots & * \\
k &
\end{pmatrix}. \]
Then $e = \frac{1}{k} b_0$, $\tau=\frac{1}{k} \tau_0$. Suppose
$\sim$ is the identity map. Then $\tau(b_i b_i)=1$ for all $i$.
So $\tau_0(b_i b_i)=k$ for all $i$.
\end{exm}

\begin{exm}
The matrix
\[ s:=\begin{tiny} \left( \begin{array}{cccccc}
-2 & 0 & 2 & -2 & 0 & -2 \\
0 & -2 & -2 & -2 & -2 & 0 \\
2 & -2 & 0 & 0 & 2 & -2 \\
-2 & -2 & 0 & 0 & 2 & 2 \\
0 & -2 & 2 & 2 & -2 & 0 \\
-2 & 0 & -2 & 2 & 0 & -2 \\
\end{array} \right) \end{tiny} \]
is the exterior square of the character table of
$\Zz/2\Zz\times\Zz/2\Zz$. By the formula of Verlinde, it
defines structure constants for a $\Zz$-based rng.
We have $\tau = -\frac{1}{4}(\tau_0+2\tau_1+\tau_5)$.
\end{exm}

\begin{exm}
Consider the matrix
\[ s:=\begin{tiny}
\left( \begin{array}{cccc}
1 & 1 & 1 & 1 \\
1 & 0 & 1 & 0 \\
1 & 1 & 0 & 0 \\
1 & 0 & 0 & 0 \\
\end{array} \right)
\end{tiny}. \]
The lattice spanned by the columns of $s$ is closed under
componentwise multiplication but it is not a $\Zz$-based rng
because there is no adequate involution $\sim$. Remark
that $s$ is not orthogonal. This is the ``$s$-matrix'' of
the monoid ring $\Zz[(\Zz/2\Zz)^2]$ where $(\Zz/2\Zz)^2$
is a monoid by componentwise multiplication.
\end{exm}

\subsection{Complex conjugation on the columns of $s$}

\begin{lem}\label{closedcomplex}
Let $s$ be the $s$-matrix of a $\Zz$-based rng $R$.
Then the set of columns $\{s_0,\ldots,s_{n-1}\}$ of $s$
is closed under complex conjugation and $s_{\tilde i}=\bar s_i$ for all $i$.
\end{lem}

\begin{proof}
The matrix $s$ is orthogonal, so $D:=s\bar s^T$ is diagonal.
Let $d$ be the vector with entries $d_i:=D_{ii}^{-1}$ and
denote by $\bar b$ the complex conjugate of a vector
$b\in\Cz^n\cong R_\Cz$.
Then we have (remember that $e=\sum_i e_i b_i$,
$\tau=\sum_i \bar e_i \tau_i$)
\[ e_i=\sum_j \bar s_{ji} d_j \]
and hence
\[ \tau(b_i \bar b_u) = \sum_j \bar e_j \tau_j(b_i \bar b_u)
= \sum_j \bar e_j \sum_l s_{li} \bar s_{lu} \bar s_{lj} d_j
= \sum_{j,m,l} s_{mj} d_m s_{li} \bar s_{lu} \bar s_{lj} d_j.\]
Since the rows $l$ and $m$ of $s$ are orthogonal by proposition
\ref{orthrel},
\[ \tau(b_i \bar b_u) = \sum_m d_m s_{mi} \bar s_{mu} = \delta_{iu} \]
which implies $\tilde b_i = \bar b_i$ for all $i$.
\end{proof}

Let $\zeta$ be a primitive $q$-th root of unity, $q\in\Nz$.
We will need the following proposition later on.

\begin{prop}\label{allmuequal}
Let $(R,B)$ be a $\Zz$-based rng with $s$-matrix $s$ with
columns of the form $s_i=\mu_i v_i$, $v_i\in C'^n$,
$C'=\{\zeta,\ldots,\zeta^q\}$, $\mu_i\in\Zz$.
Then $|\mu_i|=|\mu_j|$ for all $i,j$.
\end{prop}

\begin{proof}
Since $v_i^{-1}=\bar v_i = \frac{1}{\mu_i}s_{\tilde i}$ by lemma
\ref{closedcomplex}, for $\eta:=\tau(e)=\tau(v_i v_i^{-1})$ we get
\[ 1=\tau(b_i\tilde b_i)=\mu_i^2\tau(v_i v_i^{-1})=\mu_i^2\eta \]
i.e. $\mu_i = \pm\sqrt{\eta^{-1}}$ for all $i$.
\end{proof}

\subsection{Closed subsets}\label{closedsubsets}

\begin{defi}
Let $(R,B)$ be a $\Zz$-based rng. A subset $B'\subset B$ is
called \defst{closed}, if the $\Zz$-module spanned by $B'$ is closed
under the multiplication of $R$.
\end{defi}

\begin{exm}
The closed subsets of a character ring of a finite group $G$ are in
bijection with the normal subgroups of $G$. The subalgebra
corresponding to $N\unlhd G$ is isomorphic to the character
ring of the factor group $G/N$. This is a corollary to a theorem
of Burnside and Brauer (see \cite{mC}, 3.1.2).
\end{exm}

We prove that the subalgebra given by a closed subset is a
$\Zz$-based rng. First, we need:

\begin{lem}\label{lempotne0}
If $(R,B)$ is a $\Zz$-based rng and $b\in B$, then there
exists an $m \in \Nz$ such that $\tau(b^m)\ne 0$.
\end{lem}

\begin{proof}
Let $v$ be the column of the $s$-matrix $s$ of $R$ corresponding to
$b$. If $v^m=\sum_i \mu_i x_i$, where $x_i$ are the columns
of $s$, then $\tau(b^m)=\sum_i \bar e_i \mu_i$.
Let $w$ be the vector with the inverses of the norms of the rows of $s$
as entries, so it has positive non zero real numbers as entries.

The $X_i:=wx_i$ (componentwise multiplication) are
orthonormal, because the rows of $s$ are orthogonal.
The sum $e=\sum_i e_i x_i$ is the linear combination that
gives the vector with all entries $1$ (because it is the unit
of the algebra).

Now $w v^m = \sum_i \mu_i w x_i = \sum_i \mu_i X_i$,
and because the $X_i$ are orthonormal, we get
\[ \tau(b^m)= \sum_i e_i \mu_i
= (\sum_i \mu_i X_i\mid\sum_j e_j X_j) = (wv^m\mid w),\]
($w=\bar w$, $e=\bar e$, $(\cdot\mid\cdot)$ is the standard
inner product on $\Cz^{|B|}$).

If the basis has $n$ elements, then
the powers $v,v^2,\ldots,v^{n+1}$ of $v$ are linearly dependent.
Let $v'$ be the vector
\[ v'_i:=\begin{cases} v_i^{-1} &\mbox{if } v_i\ne 0 \\
 0 &\mbox{if } v_i=0 \end{cases} \]
and $\pi$ be the projection onto the subspace of $\Cz^{|B|}$
given by the nonzero entries of $v$, so
$\pi : u \mapsto v v' u$.
Choose $k$ minimal such that
$\pi(e)=vv',v^1,v^2\ldots,v^k$ are linearly dependent. Then there exist
$c_0,\ldots,c_{k-1}\in\Cz$ with
\[ v^{k} + c_{k-1}v^{k-1}+\cdots+c_1 v + c_0 \pi(e)=0. \]
Now suppose $\tau(b^m)=0$ for all $m>0$.
Then $( w v^m\mid w)=0$ for all $m>0$ and
in particular
\begin{eqnarray*}
0 & = & \langle w(v^{k} + c_{k-1}v^{k-1}+\cdots+c_1 v+c_0 \pi(e)),w\rangle \\
& = & \langle w c_0 \pi(e),w \rangle \\
& = & c_0 \langle \pi(w),\pi(w) \rangle,
\end{eqnarray*}
i.e. $c_0=0$.
Then
\[ v'(v^{k} + c_{k-1}v^{k-1}+\cdots+c_1 v)=
v^{k-1} + c_{k-1}v^{k-2}+\cdots+c_1 \pi(e)=0 \]
in contradiction to the minimal choice of $k$.
\end{proof}

The last proof also shows that the entries of
the $s$-matrix are algebraic over $\Qz$,
because the powers of a column are linearly dependent.
We have more:

\begin{prop}\label{algint}
The entries of the $s$-matrix of a $\Zz$-based rng are
algebraic integers over $\Zz$.
\end{prop}
\begin{proof}
If $c_0,\ldots,c_{n-1}$ are the entries of a row
of $s$, then $\Zz[c_0,\ldots,c_{n-1}]$ is finitely
generated as a $\Zz$-module, since $R$ is isomorphic
to the lattice spanned by the columns of $s$ with
componentwise multiplication, and
it is a free $\Zz$-module of rank $n$. By
Theorem \cite{jN}, I.2.2, this is equivalent to
the fact that $c_0,\ldots,c_{n-1}$ are integral
over $\Zz$.
\end{proof}

Presumably there exists an $m$ with $\tau(b^m)\in \Rz_{>0}$.
But it does not suffice to choose an $m$ such that
$\tau(b^m)\in\Rz_{<0}$ and to take the square of $b^m$.

\begin{exm}
If $R:=\Zz[\Zz/3\Zz]$ with the group elements
$b_0=1,b_1,b_2$ as a basis, then $\tilde b_1 = b_2$ and $b_1 b_2 = 1$.
We have $\tau(-b_0-b_1+b_2)=-1$ and $\tau((-b_0-b_1+b_2)^2)=
\tau(-b_0+3 b_1-b_2)=-1$. However, $\tau((-b_0-b_1+b_2)^3)=5>0$.
\end{exm}

\begin{prop}\label{RRtzB}
Let $(R,B)$ be a $\Zz$-based rng and $R'$ the subalgebra
generated by a closed subset $B'\subseteq B$.
Then $\widetilde{R'}\subseteq R'$, i.e.,
closed subsets are $\sim$-invariant.
\end{prop}

\begin{proof}
It suffices to check that $\tilde b \in B'$ for all $b\in B'$.
For this, take $b\in B'$ and $m\in\Nz$ such that $\tau(b^m)\ne 0$
(lemma \ref{lempotne0}).
If $b^{m-1}=\sum_{b'\in B'} c_{b'}b'$ for some $c_{b'}\in \Zz$, then
\[ 0 \ne \tau(b^m)=\tau(\sum_{b'}bc_{b'}b')=\sum_{b'}c_{b'}\tau(bb')=c_{\tilde b},\]
i.e. $\tilde b \in B'$ and hence $\widetilde{R'} \subseteq R'$.
\end{proof}

The $\Zz$-module spanned by the complement of a closed subset $B'$
with subalgebra $R'$ is an $R'$-module:

\begin{prop}\label{RRmodul}
If $(R,B)$ is a $\Zz$-based rng and $R'$ a subalgebra
spanned by a closed subset $B'\subseteq B$, then the $\Zz$-module
spanned by $B\backslash B'$ is an $R'$-module with respect to the
multiplication of $R$.
\end{prop}

\begin{proof}
Let $B'':=B\backslash B'$ and $b\in B', m\in B''$.
If $bm = c a + r$ with $a \in B'$, $c \in \Zz$ and
$r \in \langle B\backslash \{a\}\rangle$ then
\[ \tau(\tilde abm) = \tau(\tilde a (ca+r)) = \tau(c \tilde a a)+\tau(\tilde a
r) = c, \]
($\tau(\tilde a r)=0$).
So $\tilde ab=c\tilde m+r'$ for some suitable $r'$ with
$\tau(\tilde m r') = 0$. But since $\tilde a, b$ are in
$B'$ and $\tilde m$ is in $B''$ (proposition \ref{RRtzB}),
$c$ has to be zero. This holds for all $a\in B'$ and
$bm$ therefore lies in the span of $B''$.
\end{proof}

Now we can prove:

\begin{thm}\label{clsbisZ}
The subalgebra of a $\Zz$-based rng spanned by a closed subset
is a $\Zz$-based rng.
\end{thm}

\begin{proof}
Let $(R,B)$ be a $\Zz$-based rng and $B'\subseteq B$ be a closed subset
generating a subalgebra $R'$.
By proposition \ref{RRtzB}, we may restrict $\tau$ and $\sim$ to $R'$. It
remains to check that we have an identity element in $R'_{\Cz}=R'\otimes_{\Zz}\Cz$.
In $R_\Cz$ we have the element $e=\sum_{b\in B} e_b b$, and for
$r=\sum_{b'\in B'} \mu_{b'} b' \in R'$ we get
\[ r=er=\sum_{b\in B} e_b \sum_{b'\in B'} \mu_{b'} b'b=
\sum_{b,b''\in B}\sum_{b'\in B'} e_b \mu_{b'} N_{b'b}^{b''} b'' \]
where $\sum_{b\in B}\sum_{b'\in B'} e_b\mu_{b'}N_{b'b}^{b''}=0$ for $b''\notin
B'$ because $r\in R'$.
By proposition \ref{RRmodul}, $N_{b'b}^{b''}=0$ if $b',b''\in B'$ and $b\notin
B'$. So
\[ r=\sum_{b,b'',b'\in B'} e_b \mu_{b'} N_{b'b}^{b''} b''=
\sum_{b\in B'} e_b \sum_{b'\in B'} \mu_{b'} b'b=e'r \]
with $e':=\sum_{b\in B'} e_b b$.
\end{proof}

Closed subsets may be read of the $s$-matrix (in the same
way as character rings of factor groups of a finite group
may be found in its character table):

\begin{prop}
Let $(R,B)$ be a $\Zz$-based rng with $s$-matrix $s$ and $B'\subseteq B$ be a
closed subset. Then if $u_i$, $i\in \{0,\ldots,|B|-1\}$ are the rows of the
submatrix $u$ of $s$ consisting of the columns $s_i$, $i\in B'$, then
the matrix $u'$ with rows $\{u_i \mid 0\le i< |B|,\:\: u_i\ne 0\}$ (in any ordering)
is an $s$-matrix of the subring $R'$ spanned by $B'$.
\end{prop}

\begin{proof}
The rows of $s$ are the irreducible representations of $R$ evaluated
on $B$. Restricting them to $R'$ yields representations
$u_i$, $i\in \{0,\ldots,|B|-1\}$ of $R'$ (possibly $u_i=u_j$ for $i\ne j$,
$u_i=0$ may also occur). Now if $m:=|\{u_i \mid 0\le i< |B|, \:\: u_i\ne 0\}|$
was smaller than $|B'|$, then the rank
of $u$ would be less than $|B'|$. But this is impossible because
$s$ is invertible. On the other hand, $R'$ is a $\Zz$-based rng by
Theorem \ref{clsbisZ}, so it has $|B'|$ irreducible representations
and not more, i.e. $m=|B'|$ and $u'$ is an $s$-matrix.
\end{proof}

A converse to the previous result also holds:

\begin{prop}
If $(R,B)$ is a $\Zz$-based rng with $s$-matrix $s$ and $B'$ is a
subset of $B$ such that the set of rows $\ne 0$ of the submatrix $u$ of $s$
consisting of the columns of $s$ indexed by $B'$ has exactly $|B'|$
elements, then $B'$ is a closed subset.
\end{prop}

\begin{proof}
Let $B'\subseteq B$ be a subset as in the assumptions and
$u_i$ be the rows of the submatrix $u$ of $s$.
We partition the set of indices of rows of $u$ not equal to $0$
in a disjoint union of sets $T_1,\ldots,T_r$, $r\in\Nz$ such that
$u_i=u_j$ for $i,j \in T_l$ and
$u_i\ne u_j$ for $i\in T_l$, $j\in T_m$ with $l\ne m$
for all $i,j,l,m$. Let $w_l$ be the vector with $1$'s
at the entries $i\in T_l$ and $0$ else. Then
\[ V:=\langle w_1,\ldots,w_r \rangle \]
is a vector space of dimension $r$ and $r=|B'|$ because
there are $|B'|$ different rows in $u$.
The columns of $u$ are obviously elements of $V$. But the
product of two columns of $u$ also lies in $V$, because
all entries to a $T_l$ remain equal. Further, the rank
of $u$ is $r$ since $s$ is invertible, so $V$ is equal
to the space spanned by the columns of $u$. So the linear
decomposition $\sum_l N_{ij}^l s_l$ of the product of two columns
$s_i,s_j$ of $u$ with respect to the columns of $s$ has only
coefficients $N_{ij}^l$ not equal to $0$ for $l\in B'$ which
means that $B'$ is a closed subset.
\end{proof}

The last two propositions suggest an algorithm for a heuristic search
for closed subsets. Start with two rows of an $s$-matrix $s$ and
consider the set $B'$ of indices for which these rows have the same
entry. Then count the different rows in the submatrix consisting
of the columns of $s$ indexed by $B'$. If this number is $|B'|$
then $B'$ is a closed subset.

We do not find all closed subsets by this method in general. But
we can improve the algorithm by adding the intersection of any
two such sets and iterate this procedure until there are no further
intersections left. But there may still exist some closed sets
that we do not find with this algorithm (remark that these
subsets are rare in our applications).

\begin{exm}
The exterior square $s$ of the $s$-matrix of the group ring $\Zz[(\Zz/2\Zz)^3]$
is an $s$-matrix of a $\Zz$-based rng of rank $28$. It has a closed
subset with $16$ elements which may not be found by our algorithm.
\end{exm}

\section{Ideals and factor rings}\label{idealsandfactors}

\subsection{Pointed algebras and factor rings}

We want to define ``factor rings'' for a generalization
of $\Zz$-based rngs.

\begin{defi}
Let $R$ be a finitely generated commutative $\Zz$-algebra which is a free
$\Zz$-module of finite rank with basis $B$. We call $(R,B)$ a
\defst{pointed $\Zz$-algebra}.

We have \defst{structure constants} as in the special case
of $\Zz$-based rngs.
\end{defi}

\begin{defi}\label{facring}
Let $(R,B)$ be a pointed $\Zz$-algebra and $I$ an ideal in $R$.
If $B'\subset B$ is a subset of the basis such that
the factor ring $R/I$ is a free $\Zz$-module with
basis $B'+I$, then we call $(R/I,B'+I)$ a \defst{factor
ring} of $(R,B)$.

Let $(R,B)$ be a $\Zz$-based rng and $I$ an ideal in $R$.
If $B'\subset B$ is a subset of the basis such that
the factor ring $(R/I,B'+I)$ with basis given by $B'$ is
a $\Zz$-based rng, then we call $(R/I,B'+I)$ a \defst{factor
ring} of $(R,B)$.
\end{defi}

Remark that for most ideals, the factor module $R/I$ is not
free anymore. And even if it has no torsion, in most cases there
will be no involution $\sim$ on $R/I$ (of course this depends
on the choice of subset in $B$).

Examples for such factor rings will be given in the following
sections: There is an interesting ideal in the representation ring
of the quantum double of a finite group which yields a pointed
$\Zz$-algebra with nonnegative structure constants. Also,
each Hadamard matrix corresponds to a ring which is a factor
ring of a ring with nonnegative structure constants.

\subsection{An ideal which yields a free factor ring}

When $R$ is the representation ring of some Hopf algebra, then
in many cases there is a one dimensional ``sign'' representation
which acts via tensor product as a permutation on the irreducible
representations. We can use this to become a new pointed $\Zz$-algebra
with nonnegative structure constants:

\begin{thm}\label{ord2facalg}
Let $(R,B)$ be a $\Zz$-based rng and $d\in B$ an element of
order $2$ which acts via multiplication as a permutation on $B$. Consider
the ideal
\[ I:=\langle 1-d \rangle \unlhd R. \]
Then there is a subset $B'\subseteq B$ such that $(R/I,B'+I)$ is
a pointed $\Zz$-algebra with nonnegative structure constants.
\end{thm}

\begin{proof} By assumption,
multiplication by $d$ is a bijection on the basis.
Since $d^2=1$, we get a partition of $B$ into equivalence classes
$\{ b, db \}$, equivalent elements are equal in $R/I$. Let $B'$ be a set of
representatives.
The $\Zz$-module $R':=R/I$ is free, and the set
$\{ a+I \mid a \in B'\}$ is a basis: assume that $R'$ has a
torsion element $\bar w$, $w\in R$, so $n\bar w=0$ for some $n\in \Nz$.
Then $nw\in I$, i.e. $nw$ may be written as $r(1-d)$ for an
$r=\sum_{i} a_i b_i \in R$, $a_i\in\Zz$, if the $b_i$ are the elements of $B$.
Let $\hat i$ be the index such that $b_{\hat i}=b_i\cdot d$.
Then
\[ nw=r(1-d)=\sum_i a_i b_i (1-d)= \sum_i(a_i-a_{\hat i})b_i, \]
and hence each $(a_i-a_{\hat i})$ is divisible by $n$. Let us write
$(a_i-a_{\hat i})=: n c_i$ with $c_i \in \Zz$.
Moreover,
\[ \sum_{b_i\in B}(a_i-a_{\hat i})b_i =
\sum_{b_i\in B'} (a_i-a_{\hat i})(b_i-b_{\hat i}) =
n \sum_{b_i\in B'} c_i (b_i-b_{\hat i}) \in n I.\]
But a $\Zz$-based rng is a free $\Zz$-module. Therefore
$n r_1=n r_2$ implies $r_1=r_2$ for all $r_1,r_2\in R$.
Because of $nw\in nI$, we get $w\in I$, i.e. $\bar w=0$.

The structure constants $\tilde N_{*,*}^*$ of $R/I$
for $b_i,b_j,b_k\in B'$ are
\[ \tilde N_{\BAR{b_i},\BAR{b_j}}^{\BAR{b_k}} =
N_{b_i,b_j}^{b_k}+
N_{b_i,b_j}^{b_{\hat k}}\]
and therefore they are nonnegative.
\end{proof}

\subsection{$\Zz$-based rngs as factor rings}

An interesting task is to find for a given $\Zz$-based rng $R'$
a $\Zz$-based rng $R$ with nonnegative structure constants such that
$R'$ is a factor ring of $R$.

Let $W$ be the set of one dimensional subspaces of $\Cz^n$ spanned
by products of columns of $s$,
\[ W:=\Big\{ \langle \prod_{\nu=1}^r s_{i_\nu} \rangle \mid
0\le i_1,\ldots,i_r<n,\: r\in\Nz \Big\}. \]
Denote by $K$ the field extension of $\Qz$ given by the entries
of $s$, by $\Oc$ its ring of algebraic integers (see proposition
\ref{algint}) and by
$C$ the set of all roots of unity in $\Oc$ and $0$,
so $C=\{ 0,\zeta,\ldots,\zeta^q\}$ for some $q\in\Nz$
and primitive $q$-th root of unity $\zeta$.

\begin{lem}\label{lem1}
If $W$ is finite, then there are $\mu_0,\ldots,\mu_{n-1}\in\Oc$ such that
$s_{ij}\in\mu_j C$ for all $0\le i,j<n$.
\end{lem}

\begin{proof}
Choose $i_1,i_2,j$ with $s_{i_2,j}\ne 0$ and $\mu:=s_{i_1,j}s_{i_2,j}^{-1}$.
Since $W$ is finite, the powers $s_j^{k_1}, s_j^{k_2}$ of the $j$-th column
are linearly dependent for some $k_1,k_2\in\Nz$, so we have
\[ s_{i_1,j}^{k_1}+\lambda s_{i_1,j}^{k_2}=0,\quad
s_{i_2,j}^{k_1}+\lambda s_{i_2,j}^{k_2}=0 \]
for $0\ne\lambda\in\Cz$. Hence
\[ \mu^{k_1}s_{i_2,j}^{k_1}+\lambda \mu^{k_2}s_{i_2,j}^{k_2}=0,\quad
\mu^{k_1}s_{i_2,j}^{k_1}+\lambda \mu^{k_1}s_{i_2,j}^{k_2}=0 \]
which implies $(\mu^{k_1}-\mu^{k_2})\lambda s_{i_2,j}^{k_2}=0$, i.e.
$\mu$ is a root of unity, $\mu\in C$.
\end{proof}

If for example all entries of $\frac{1}{\mu}s$ are roots of unity
for some $\mu\in\Nz$, then the columns of $\frac{1}{\mu}s$ generate
(multiplicatively) a finite group.
In this case $R$ is a factor ring of a subring of $\mu$ times this group ring.
More generally, for such a construction we need at least the finiteness of the
set $W$.

\begin{thm}\label{fannsc}
Let $(R,B)$ be a $\Zz$-based rng with $s$-matrix $s$.
Assume that the columns $s_0,\ldots,s_{n-1}$ are of the
form $s_i=\mu_i v_i$ with $v_i\in C^n$ and $0\ne\mu_i\in\Zz$.
Then $R$ is a factor ring of a pointed $\Zz$-algebra
with nonnegative structure constants.
\end{thm}

\begin{proof}
For all $a\in C$ we have $-a\in C$, so without loss of generality we
may choose $\mu_i> 0$ for all $i$.
Let $G$ be the multiplicative semigroup generated by the
columns of $s$. Define a partial order $\le$ on $\Cz^n$ by
\[ v \le w \quad \Leftrightarrow \quad
\Nz v \subseteq \Nz w \]
for $v,w\in \Cz^n$. Let $M$ be the set of maximal elements of $G$ with
respect to this order (they exist, because all entries of an element
of $G$ are of the form $\lambda\xi$ with $\lambda\in\Zz$ and $\xi$ a
nonnegative power of $\zeta$). The multiplicative semigroup $H$ generated
by $v_0,\ldots,v_{n-1}\in C^n$ is finite because $C^n$ is finite.
For each $v\in H$, there may be several elements $w_i$ of $M$ with $w_i\le v$.

Let $y_i$ be the number of elements of the multiplicative subsemigroup
of $H$ generated by $v_i$. So
\[ H = \Big\{ \prod_{i}^r v_i^{a_i} \mid 0< a_i \le y_i \mbox{ for all }
i \Big\}. \]
Let $v$ be any element of $G$, we write it in the form
$v=\prod_i \mu_i^{a_i} v_i^{a_i}$.
Define $c_i$ to be the number in $0,\ldots,y_i-1$ with
$ c_i \equiv a_i \:\:(\modu y_i)$.

Then $v=\prod_i \mu_i^{a_i-ci} t$ for $t=\prod_i \mu_i^{c_i} v_i^{c_i}$.
But $t\ge v$, so for all elements of $G$ we find a greater element
in the finite subset
\[ U = \Big\{ \prod_{i}^r s_i^{a_i} \mid 0< a_i \le y_i \mbox{ for all }
i \Big\}, \]
in particular, $M\subseteq U$ and it is finite. We need a set
of generators for our $\Zz$-algebra. For this we take $M$ and for
all $a_1 v, a_2 v\in M$, $v\in H$, $a_1,a_2\in\Nz$ we add the
element $\gcd(a_1,a_2)v$. We iterate this until no more new elements
appear and denote the resulting set by $M''$. (This iteration terminates
because $M$ is finite.)

It can happen that the product of two elements of $M''$ is
greater than all elements of $M''$, so we add all these products
and do the above iteration for $M''$ instead of $M$ again.
Since the set of elements of
$\Nz H$ that are greater or equal to all elements of $M$ is finite,
this process terminates. We need this construction because we
have to show that we do not leave the $\Zz$-lattice spanned by
the columns of $s$:

Since the structure constants are integers, all elements
of $G$ are in this lattice; the only critical point is when we
add $\gcd(a_1,a_2)v$. So as above, let $a_1 v, a_2 v\in M$, $v\in H$,
$a_1,a_2\in\Nz$ and $d:=\gcd(a_1,a_2)$. Suppose
\[ a_1 v=\sum_i q_i s_i, \quad a_2 v= \sum_i r_i s_i, \]
for some $q_i, r_i\in \Zz$. Consider $d v$, which is the
new element we want to add to $M$. We have
\[ dv = \sum_i \frac{q_i d}{a_1}s_i = \sum_i \frac{r_i d}{a_2}s_i \]
and since the $s_i$ form a basis, we get
$\frac{q_i d}{a_1}=\frac{r_i d}{a_2}$ for all $i$. This means
$r_i\frac{a_1}{a_2}=q_i\in\Zz$, so $\frac{a_2}{d}$ divides $r_i$
and $dv$ is in the lattice.

Let $M'$ be the finite set of now uniquely determined maximal
elements of $M''$. The product of two elements of $M'$ is not greater
than any element of $M'$. Remark that the columns of $s$ are included
in $M'$, because they are not smaller than any linear combination
with integer coefficients of the $s_i$ and every element of $M'$
is such a linear combination. The $s_i$ have already been
in $M$ at the beginning.
Let $G'$ be the multiplicative semigroup generated by $M'$.

Let $R'$ be the free $\Zz$-module with basis $M'$
(view $M'$ as a ``formal'' basis, so $R'$ has rank $|M'|$). We define
the multiplication on $R'$ in the following way. If $v,w\in M'$, then
as elements of $G'$ we have $v\cdot w=u\in G'$ and there exists
a unique $x\in M'$ with $x\ge u$, so $u=\mu x$ for some $\mu\in\Nz$.
We set $v\cdot w:=\mu x$ in $R'$. This multiplication is
associative, because the multiplication of $G'$ is associative.
Also, the structure constants of $R'$ are nonnegative because
$\mu\in\Nz$.

To get $R$ as a factor rng of $R'$, we only need to see that the
elements of $M'$ are in the $\Zz$-lattice spanned by the columns
of $s$, which we have proved above. The ideal $I$ in $R'$ is generated
by $x-\sum_i q_i s_i$ if $\sum_i q_i s_i$ is the linear decomposition
of $x$ in $R$ with respect to the columns of $s$.
We take the set $B'$ of columns of $s$ as a subset of $M'$.
Then $(R,B)\cong (R'/I,B'+I)$.
\end{proof}

\begin{rem}
The ring $R'$ of the last theorem is a subring of
the semigroupring $\Zz[H]$ of $H$.
\end{rem}

Now we can conclude:

\begin{prop}\label{factorhad}
Let $(R,B)$ be a $\Zz$-based rng with $s$-matrix $s$, $\zeta$
a complex $q$-th root of unity and $C':=\{\zeta,\ldots,\zeta^q\}$.
Assume that the columns $s_0,\ldots,s_{n-1}$ are of the
form $s_i=\mu_i v_i$ with $v_i\in {C'}^n$ and $0\ne\mu_i\in\Zz$.
Then $R$ is a factor ring of a subring with nonnegative structure
constants of $k$ times a group ring for $k=|\mu_i|$ for all $i$.
\end{prop}

\begin{proof}
We use the notation of the proof of Theorem \ref{fannsc}.
Since $0\notin C'$, the set $H$ is a group.
By proposition \ref{allmuequal} we have $|\mu_i|=|\mu_j|$
for all $i,j$ and since without loss of generality $\mu_i>0$
for all $i$, they are all equal to some $k\in\Nz$.
So all elements of $M'$ lie in $k\Zz[H]$, in particular
$R'\le k\Zz[H]$.
\end{proof}

\begin{quest}
Assume that the above set of spaces $W$ is finite and that
$K=\Qz[\zeta]$ for some root of unity $\zeta$.
Is it true that $R$ is a subalgebra of a factor ring of a pointed $\Zz$-algebra
with nonnegative structure constants?
\end{quest}

To prove this, it would be enough to construct a matrix
$s'$ starting from $s$ satisfying the assumptions of Theorem \ref{fannsc}
by using the fact that $\Oc=\Zz[\zeta]$ is a pointed $\Zz$-algebra
with basis given by some powers of $\zeta$.

\section{Examples: some modular tensor categories}

\subsection{Quantum doubles of finite groups}

Let $G$ be a finite group and $k$ an algebraically closed field
of characteristic $0$. The quantum double $D(G)$ of $G$ is a Hopf
algebra isomorphic to $F(G)\otimes k[G]$ as a vector space, where
$F(G)$ is the algebra of functions on $G$ into $k$ and $k[G]$ is
the group algebra. Its algebra and coalgebra structure will not
be used here, see \cite{bBaK}, 3.2 for more details or
\cite{CGR} for a twisted version.

The category of finite dimensional representations
$\Rep_f(D(G))$ of $D(G)$ is a modular tensor category
(see \cite{bBaK}). Its Grothendieck ring $R(D(G))$ is a fusion algebra,
so it is a $\Zz$-based ring with nonnegative structure constants.
There is an explicit formula for the corresponding $s$-matrix.

The irreducible representations of $D(G)$ are parametrized by
pairs $(\bar g,\chi)$ where $\bar g$ is a conjugacy class of $G$
and $\chi$ an irreducible character of the centralizer of $g$
in $G$. We choose one representative $g$ for each class $\bar g$
and denote the irreducible modules by $(g,\chi)$.
So we may identify the basis elements of $R(D(G))$ with
pairs $(g,\chi)$.

Let $e$ denote the identity element of $G$ and $R:=R(D(G))$.
Assume the existence of an irreducible $D(G)$-module $(e,\chi)$
with $(e,\chi)\cdot (e,\chi)=(e,1)$, where ``$\cdot$'' is the
tensor product in $R$. This just means that $\chi^2=1$,
so for example if $G=S_n$, we can choose the sign character for
$\chi$. Let
\[ I := \langle (e,1)-(e,\chi) \rangle \unlhd R, \]
so in $R/I$ we identify each basis element $(g,\psi)$ with
$(e,\chi)(g,\psi)$.

Assume that $(e,\chi)$ acts via multiplication as a
permutation on the basis elements of $R(D(G))$. Then by theorem
\ref{ord2facalg}, the factor ring $R(D(G))/I$ is a pointed
$\Zz$-algebra with nonnegative structure constants.

\begin{exm}
Let $G=S_3$. The $s$-matrix of $R(D(G))/I$ is
\[ \left( \begin{array}{cccccc}
1 & 2 & 3 & 2 & 2 & 2 \\
1 & 2 & -3 & 2 & 2 & 2 \\
1 & 2 & 0 & -1 & -1 & -1 \\
1 & -1 & 0 & -1 & -1 & 2 \\
1 & -1 & 0 & -1 & 2 & -1 \\
1 & -1 & 0 & 2 & -1 & -1 \\
\end{array} \right). \]
\end{exm}

\subsection{Affine Kac-Moody algebras}

let ${\frak g}(A)$ be an affine Kac-Moody algebra of arbitrary
untwisted type $X_l^{(1)}$ or $A_{2l}^{(2)}$ and
denote the fundamental weights by $\Lambda_0,\ldots,\Lambda_l$.

For each positive integer $k$, let $P_+^k$ be the finite set
\[ P_+^k:=\Big\{ \sum_{j=0}^l \lambda_j\Lambda_j\mid
\lambda_j\in\Zz, \lambda_j\ge 0, \sum_{j=0}^l a_j^\vee\lambda_j=k \Big\}. \]
Kac and Peterson defined a natural $\Cz$-representation of the group
$\mbox{SL}_2(\Zz)$ on the subspace spanned by the affine characters of
${\frak g}$ which are indexed by $P_+^k$.
The image of $\tiny \begin{pmatrix} 0&-1\\ 1&0 \end{pmatrix}$ under
this representation is determined in Theorem $13.8$ of \cite{vK}.
This is the so-called \defst{Kac-Peterson matrix}:
\[ S_{\Lambda,\Lambda'} = c\sum_{w\in W^\circ} \det(w)
\exp\left({-\frac{2\pi \mathrm{i} (\overline\Lambda+\bar\rho\mid
 w(\overline\Lambda'+\bar \rho))}{k+h^\vee}}\right), \]
where $\Lambda,\Lambda'$ runs through $P_+^k$, $(\cdot\mid\cdot)$ is
the normalized bilinear form of chapter $6$ of \cite{vK},
$W^\circ$ is the Weyl group of $\frak g^\circ$ and $c\in\Cz$ some constant
which is unimportant for us, since we want to use the matrix in Verlinde's
formula (section \ref{smat}, (\ref{VF})).

Each of these matrices defines a $\Zz$-based rng (even a fusion algebra).
A classification of the matrices belonging to type $X_l^{(1)}$ up to isomorphism
was given by Gannon in \cite{tG}.

\begin{exm}
Lusztig \cite{gL} gives an interpretation of the Fourier matrix
for dihedral groups via the above fusion algebra of type $A_1^{(1)}$,
level $k$. Denote this algebra by $R$. Then the Fourier matrix is
the $S$-matrix of a factor ring of the algebra belonging to a closed
subset of $R\otimes R$.
\end{exm}

\begin{exm}
Geck and Malle \cite{mGgM} use a similar construction (they choose another ideal)
to become the Fourier matrix for the dihedral group with non-trivial
automorphism. The resulting algebra has negative structure constants,
so it is a $\Zz$-based ring.
\end{exm}

\begin{exm}
The author \cite{mC} has interpreted many of the Fourier matrices associated to
the exceptional spetses as factor rings of algebras corresponding to
Kac-Peterson matrices.

For example, let $R$ be the affine ring corresponding to the Kac-Peterson
matrix of type $B_3^{(1)}$ and level $2$. The algebras corresponding to
the Fourier matrix of some families of the exceptional complex reflection
group $G_{24}$ are isomorphic to $R'/I$ for some ideal $I$,
where $R'$ is a subring (to a closed subset) of $R\otimes Z_{4}$ ($Z_{4}$
is the group ring of the cyclic group with $4$ elements).
\end{exm}

\section{Example: Hadamard matrices}

\subsection{General case}

Let $s$ be a Hadamard matrix of \defst{order} $4k$, i.e.
\[ s\in\{\pm 1\}^{4k\times4k}, \quad
   s s^T=4kI, \quad k\in\Nz, \]
where $I$ denotes the identity matrix. We assume $k>1$ and $s_{i,1}=1$ for all
$i$.

\begin{prop}\label{ZLattice}
The $\Zz$-lattice $R$ spanned by the columns of $k\cdot s$ is closed under
componentwise multiplication, i.e., if $b_1,\ldots,b_{4k}$ are the columns of
$ks$, then
\[ b_i b_j = \sum_m N_{ij}^m b_m \]
for some $N_{ij}^m\in \Zz$. We have
\[ N_{ij}^m = \frac{1}{4} \sum_l s_{li} s_{lj} s_{lm} \]
for all $i,j,m$.
\end{prop}

\begin{proof}
Choose $i,j,m$ such that $|\{1,i,j,m\}|=4$. Then without loss of generality
(after permuting some rows), the transposed of the columns $i,j,m$ will be
\[ \begin{array}{cccccccc}
- \ldots - & - \ldots - & - \ldots - & - \ldots - &
+ \ldots + & + \ldots + & + \ldots + & + \ldots + \\
- \ldots - & - \ldots - & + \ldots + & + \ldots + &
- \ldots - & - \ldots - & + \ldots + & + \ldots + \\
\underbrace{- \ldots -}_{a} & \underbrace{+ \ldots +}_{k-a} &
\underbrace{- \ldots -}_{k-a} & \underbrace{+ \ldots +}_{a} &
\underbrace{- \ldots -}_{k-a} & \underbrace{+ \ldots +}_{a} &
\underbrace{- \ldots -}_{a} & \underbrace{+ \ldots +}_{k-a}
\end{array} \]
for some $1\le a < k$,
where ``$+$'', ``$-$'' stands for ``$1$'' respectively ``$-1$''.
That $a$ and $k-a$ appear several times comes from the mutual orthogonality
of the three columns:
Column $i$ and $j$ must have exactly $k$ common $-1$'s because they are
orthogonal. Suppose that column $m$ has $a$ $-1$'s at these positions. Then it
has to have $k-a$ $-1$'s where column $i$ is $-1$ and column $j$ is $1$
because it is orthogonal to column $i$, and so on.

Counting together, the product of the three columns is a vector with $4a$
$-1$'s. This yields $\sum_l s_{li} s_{lj} s_{lm} \equiv 4k-8a \equiv 0 \:\: (\modu 4)$.
If one of the indices $i,j,m$ is $1$ or if two of them are equal, then
$\sum_l s_{li} s_{lj} s_{lm}$ reduces to orthogonality or to counting $-1$'s in
a column.

So the numbers $N_{ij}^m$ defined above are integers. That these numbers are the
structure constants of the lattice spanned by the columns of $s$ follows from
$s^{-1}=\frac{1}{4k}s^{T}$ and Verlinde's formula (\ref{VF}).
\end{proof}

\begin{rem}
Theorem \cite{wdW}, 9.9 states that $p_{ijlm}\equiv 4k \:\:(\modu 8)$, where
\[ p_{ijlm}=\left|\sum_q s_{qi} s_{qj} s_{ql} s_{qm} \right|. \]
In particular for $m=1$ we get the absolute values of $4$ times the structure constants.
So the last proposition is a corollary of this Theorem. The \defst{profile}
of a Hadamard matrix is the map $\pi : \Nz \rightarrow \Zz$, where $\pi(q)$ is
the number of sets $\{i,j,l,m\}$ of four distinct columns such that $p_{ijlm}=q$.
Profiles are used to test if Hadamard matrices are not equivalent and
they also give some results about the number of equivalence classes.
\end{rem}

\begin{cor}
A Hadamard matrix gives rise to a $\Zz$-based rng $R$.
\end{cor}

\begin{proof}
By the previous proposition, we may associate a $\Zz$-algebra to each Hada\-mard
matrix. It has a basis consisting of elements $b_i$ with $b_i^2=kb_1$. Also,
$N_{i,j}^1=0$ for $i\ne j$.
So it is a $\Zz$-based rng with trivial involution $\sim$ and is of the
same form as example \ref{exkkk}.
\end{proof}

\begin{defi}
We call the $\Zz$-algebra $R$ of proposition \ref{ZLattice} the \defst{$\Zz$-based
rng} or the \defst{ring} associated to the Hadamard matrix $s$.
\end{defi}

\begin{lem}
If there are $i,j,l$ such that $b_i b_j=kb_l$ and $i,j,l\ne 1$ then $k$ is even
and $k\cdot \Zz[\Zz/2\Zz\times \Zz/2\Zz]$ is a subalgebra of $R$.
\end{lem}

\begin{proof}
By the assumptions, $\{b_1,b_i,b_j,b_l\}$ is a closed subset.
Up to the factor $k$, this set forms a group isomorphic to
$\Zz/2\Zz\times \Zz/2\Zz$. When $k$ is odd, such a closed subset
can not exist, see corollary \ref{hadclsts}.
\end{proof}

To describe the structure constants $N_{ij}^l$ of the ring of a Hadamard
matrix in another way, we introduce the following notation. Let
\[ \xi_i := \{ j \mid s_{ji} = -1 \} \]
and $\Delta$ denote the symmetric difference of sets.

\begin{lem} For $i,j$ with $|\{1,i,j\}| = 3$ we have
\[ |\xi_i|=|\xi_i\Delta \xi_j| = 2|\xi_i \cap \xi_j|=2k .\]
\end{lem}

\begin{proof} Since $s_{i1}=1$ for all $i$, we have $|\xi_i|=|\xi_j|=2k$
for $i,j\ne 1$. If in addition $i\ne j$ then
$|\xi_i\Delta \xi_j|=2k$ since $ss^T$ is diagonal. Hence $|\xi_i\cap \xi_j|=k$.
\end{proof}

\begin{lem} For $i,j,l$ with $|\{1,i,j,l\}| = 4$ we have
\[ |\xi_i\Delta \xi_j\Delta \xi_l| = 4|\xi_i \cap \xi_j \cap \xi_l|
\quad \mbox{and} \quad
|\xi_i \cap \xi_j \cap \xi_l|= |\xi_i \backslash (\xi_j \cup \xi_l)|.\]
\end{lem}

\begin{proof} Let $a:=|\xi_i \cap \xi_j \cap \xi_l|$.
The proof of proposition \ref{ZLattice} gives the first equality.
Further,
\[ |\xi_i|=k+|\xi_i\cap \xi_j|=|\xi_i\backslash(\xi_j \cup \xi_l)|+
|\xi_i\cap \xi_j|+|\xi_i\cap \xi_l|-a, \]
so $2k=|\xi_i\backslash(\xi_j \cup \xi_l)|+k+k-a$.
\end{proof}

Thus for the structure constants we get
(this is also a consequence of the proof of Proposition \ref{ZLattice}):

\begin{lem}\label{capdiff}
For $i,j,m$ with $|\{1,i,j,m\}| = 4$ we have
\[ N_{ij}^m = k-\frac{1}{2}|\xi_i\Delta \xi_j\Delta \xi_m|
= k-2|\xi_i \cap \xi_j \cap \xi_m|.\]
\end{lem}

\subsection{Odd $k$}

The most important open question concerning Hadamard matrices is certainly the
existence of such matrices in all dimensions $4k$.
If $k$ is even, then we can take a matrix of size $2k \times 2k$
and tensor it with a $2 \times 2$-matrix. So from now on, $k$ will
always be odd.

From Lemma \ref{capdiff}, we then get:

\begin{cor}\label{strcnotzero}
Let $k$ be odd.
If $i,j,m$ are such that $|\{1,i,j,m\}| = 4$, then
\[ N_{ij}^m \equiv 1 \:\: (\modu 2). \]
In fact, we know even more: $N_{ij}^m \in \{-k+2,-k+4,\ldots,k-2\}$.
\end{cor}

Corollary \ref{strcnotzero} gives us all the information needed to determine
the closed subsets of $R$:

\begin{cor}\label{hadclsts}
If $k$ is odd, then the closed subsets of $R$ are exactly the sets
$\{b_1,b_i\}$ for $1\le i \le 4k$ and the set of all basis elements.
\end{cor}

\begin{proof}
Suppose $B'$ is a closed subset with two different elements $b_i$ and $b_j$,
both unequal to $b_1$. Then by corollary \ref{strcnotzero},
$N_{ij}^m\ne 0$ for all $m$ with $|\{1,i,j,m\}|=4$. But $B'$ is closed,
so $b_m\in B'$ for all $m\ne 1$. Further, $b_i^2=kb_1$ which implies
$b_1\in B'$.
\end{proof}

Since the ring of a Hadamard matrix is a $\Zz$-based rng, Proposition
\ref{orthrel} applies:

\begin{prop}\label{exhad}
Let $(R,B)$ be a $\Zz$-based rng of rank $4k$.
Assume that $\sim$ is the identity on $B$, that
$k\cdot 1 = b_1 \in B$ and that
\[ N_{ii}^j = \delta_{1,j} k \]
for all $i,j$.
Then there exists a Hadamard matrix of size $4k \times 4k$.
\end{prop}

\begin{proof}
By proposition \ref{orthrel}, the $s$-matrix $s$ of $R$ has orthogonal
rows. Multiplication of elements
of $B$ corresponds to componentwise multiplication of columns of $s$. Hence
$N_{ii}^j = \delta_{1,j} k$ means $s_{li}^2 = s_{l1} k = k^2$ for all $i,l$.
So the entries of $s$ have to be equal to $\pm k$.
\end{proof}

The last proposition is a good motivation to examine the properties of the
structure constants. By Corollary \ref{strcnotzero} we already know what the structure
constants are modulo $2$. If it is possible to lift them to the $2$-adic
numbers $\Zz_2$, then the existence of a Hadamard matrix also follows.
Of course, the same holds for the entries of a Hadamard matrix:
we know them modulo $2$ and cannot lift them to $\Zz_2$.

The matrix $N_i:=(N_{ij}^m)_{j,m}$ describes the operation of the $i$-th column on
the other columns by componentwise multiplication
($N_i$ is the image of $b_i$ in the regular representation of $R$). Since we have
$b_i^2=kb_1$, the square of this matrix is a scalar matrix: $N_i^2=k^2I$. So we
get:

\begin{lem} $\sum_{m} (N_{ij}^m)^2 = k^2$ for all $i,j$.
\end{lem}

And using this, it is easy to prove (where the $r$-th triangular number is
defined to be $\frac{r(r+1)}{2}$):

\begin{prop}
For given $i,j$ with $|\{1,i,j\}|=3$, the number of different possible multisets
for the absolute values of the structure constants $|N_{ij}^m|$,
$m\notin\{1,i,j\}$ is less or equal to
the number of partitions of the $\frac{k-3}{2}$-th triangular number into
nonzero triangular numbers.
\end{prop}

\begin{proof}
We have $4k-3$ indices $m$ such that the structure constants $N_{ij}^m$
are unequal to $0$ and in this case they are odd. We are
considering multisets, so the number we are interested in is the number
of elements of
\[ \{ (i_1,\ldots,i_{4k-3}) \mid
1\le i_1 \le \ldots \le i_{4k-3}, \quad
2 \nmid i_m \:\: \forall m,\quad
\sum_{m=1}^{4k-3}i_m^2 = k^2 \} \]
\[ = \{ \bar i\in\Zz^{4k-3} \mid
0\le i_1 \le \ldots \le i_{4k-3}, \quad
\sum_{m=1}^{4k-3}(2i_m+1)^2 = k^2 \} \]
\[ = \{ \bar i\in\Zz^{4k-3} \mid
0\le i_1 \le \ldots \le i_{4k-3}, \quad
\sum_{m=1}^{4k-3} i_m^2+i_m
= \lambda^2+\lambda \} \]
for $\lambda:=\frac{k-3}{2}$.
\end{proof}

Proposition \ref{exhad} may be understood as follows. The associativity
of $R$ is equivalent to the fact that the matrices $N_i$ commute.
Since they are semisimple, with eigenvalues $\pm k\in\Zz$,
they are simultaneously diagonalizable. Commutativity of $R$
corresponds to the symmetry of the $N_i$.

Conversely, for the same reason, if a Hadamard matrix is given,
then we may recover it from the structure constants of its ring.
And if $k\equiv 1 \:\:(\modu 3)$, then we may
even reduce them modulo $3$. Indeed, over $\Fz_3$ the matrices
still commute, the eigenvalues are $\pm 1$ and $N_i^2=I$ for
all $i$. So they are still simultaneously diagonalizable which
yields the Hadamard matrix over $\Fz_3$ without loss
of information.

\begin{prop}
If $k\equiv 1 \:\:(\modu 3)$ then it is possible
to compute the Hadamard matrix from the structure constants modulo $3$
of its ring.
\end{prop}

Taking the structure constants of a Hadamard matrix modulo $2$ yields an
$\Fz_2$-algebra (this is trivial, but it should be stated for the sake of
completeness).

\begin{cor}
The vector space $\Fz_2^{4k}$ with basis $v_1,\ldots,v_{4k}$ and multiplication
defined by $v_i v_j = \sum_m N_{ij}^m v_m$ where
\[ N_{ij}^1=N_{1i}^j=N_{i1}^j=\delta_{ij},\quad N_{ij}^m=1 \]
for $|\{1,i,j,m\}|=4$, is an associative commutative algebra.
\end{cor}

We return to Hadamard matrices over $\Zz$.
The submatrix $\hat N_i$ of $N_i$ corresponding to the indices not equal to
$1,i$ ($i\ne 1$) is a square matrix with entries in $\Zz$ which has zeroes
only on its diagonal. Moreover, it has orthogonal rows,
\[ \hat N_i {\hat N_i}^T = k^2 I. \]
This is a generalization of weighing matrices or more precisely
of conference matrices (see \cite{wdW}, p. 145).

\begin{prop}\label{8k4}
Let
\[ W_i:=\begin{pmatrix}
\hat N_i+I & \hat N_i-I \\
\hat N_i-I & -\hat N_i-I
\end{pmatrix}, \]
$i\ne 1$. Then $W_i\cdot W_i^T = (2k^2+2) I.$
(Where `$I$' always denotes the appropriate identity matrix.)
\end{prop}

So $W_i$ is an orthogonal $(8k-4)\times(8k-4)$-matrix with entries in
$\Zz\backslash\{0\}$.

\begin{exm}
Let $k=3$, so $s$ is a $12\times 12$ Hadamard matrix.
The matrices $W_i$ of proposition \ref{8k4} are then $20\times 20$
Hadamard matrices. Note that this is the only case in which
the $W_i$ have entries in $\{\pm 1\}$.
\end{exm}

\subsection{Hadamard matrices as factor rings}

By proposition \ref{factorhad}, the ring of a Hadamard matrix
is a factor ring of a ring with nonnegative structure
constants. We want to look at this ring more explicitely.

Denote the cyclic group with two elements by $Z_2$.
The ring $R$ of a Hadamard matrix of size $4k\times 4k$ is a factor
ring of a subring of $k\cdot R'$, where $R'$ is the group algebra $\Zz[Z_2^{4k-2}]$.
We can see this in the following way.

Let $v$ be the matrix defined by $v_{ij}:=\frac{1-s_{ij}}{2}\in\Fz_2$
and denote the columns of $v$ by $v_1,\ldots,v_n$. Remember that
we have $s_{i,1}=1$ for all $i$, so $v_1=0$.

\begin{prop}
The rank of $v$ is less or equal to $4k-2$.
\end{prop}
\begin{proof}
The first column is $0$, so it suffices to give a linear
combination of the other columns to prove the claim.
But the number of $-1$'s in each row of $s$ is even, so
$\sum_{i=2}^n v_i = 0$.
\end{proof}

So the vector space spanned by the columns of $v$ has dimension
at most $4k-2$, as an abelian group it is isomorphic to a subgroup of
$Z_2^{4k-2}$.
Componentwise multiplication of column vectors of $s$ corresponds
to addition of column vectors of $v$.

Let $G$ be the multiplicative semigroup generated by the columns of
$ks$. Then $\Zz[G]$ is a subalgebra of $\Zz[Z_2^{4k-2}]$.

By identifying each vector of $G$ with
its linear combination with respect to the columns of $ks$
(this is a $\Zz$-linear combination, because the structure constants
are integers),
we get an ideal $I$ in $\Zz[G]$ such that $R\cong \Zz[G]/I$.

This proves:

\begin{prop}
The ring $(R,B)$ of a Hadamard matrix $s$ of size $4k\times 4k$ is a factor
ring of the algebra $\Zz[H]$
\[ R \cong \Zz[H]/I, \]
where $H$ is the multiplicative semigroup spanned by the columns of $s$.
$\Zz[H]$ is a subalgebra of the group algebra $\Zz[Z_2^{4k-2}]$.
The ideal $I$ is given by
\[ I:=\langle w-\sum_i \mu_{w,i} b_i \mid w \in G \rangle \]
where $G$ is the multiplicative semigroup spanned by the columns of $ks$,
and the $\mu_{w,i}$ are given by the linear decomposition of
$w$ with respect to the columns of $ks$.
\end{prop}

An interesting question is: What are the ideals of different
Hadamard matrices (of the same size)?
If we do the same construction for $\hat R:=\Zz[Z_2^{4k}]$, then
in $\hat R_\Cz=\hat R \otimes_\Zz \Cz$, we always get the same ideal.
The proof looks complicated, but it is completely straightforward.

\begin{prop}
In the algebra $\hat R_\Cz=\Cz[Z_2^{4k}]$, the above
ideals are equal for all Hadamard matrices.
\end{prop}

\begin{proof}
The ideals of $\hat R_\Cz$ are easy to describe. For this, it
suffices to give a basis of primitive idempotents; $\hat R_\Cz$ is
isomorphic to $\Cz^{2^n}$, $n=4k$ with componentwise multiplication.
Then it is obvious that the ideals are in bijection with
the subsets of $\{1,\ldots,2^n\}$.

We identify $Z_2$ with $\Zz^\times$ and view the columns
$s_1,\ldots,s_n$ of $s$ as elements of $Z_2^n$, so there
exists a map $\tau : \{1,\ldots,n\}\rightarrow\{1,\ldots,2^n\}$
such that $s_i=b_{\tau(i)}$ if
$b_1,\ldots,b_{2^n}$ are the elements of $Z_2^n$. Define
\[ \sigma : Z_2^n \rightarrow \Fz_2^n, \quad
 (\varepsilon_1,\ldots,\varepsilon_n) \mapsto
(\delta_{\varepsilon_1,-1},\ldots,\delta_{\varepsilon_n,-1}).\]
The primitive idempotents are given by the character table $T$ of
$Z_2^n$ which is the $n$-th tensor power of the character
table of $Z_2$. More explicitly,
\[ T_{i,j}=(-1)^{(\sigma(b_i)\mid \sigma(b_j))} \]
where $(\cdot\mid\cdot)$ is the standard inner product on $\Fz_2^n$.
So the primitive idempotents are
\[ c_i := \frac{1}{2^n}\sum_j T_{i,j} b_j \]
for $i=1,\ldots,2^n$ and we also get $b_i=\sum_j T_{i,j} c_j$.
Now in ${\Cz^\times}^n$, we have decompositions
\[ b_i = \sum_{j=1}^n \lambda_{ij} s_j \]
for some suitable $\lambda_{i,j}$. Denote
$\lambda_i:=(\lambda_{i1},\ldots,\lambda_{in})$. Then
\[ \lambda_i= s^{-1} b_i = \frac{1}{n} s^T b_i \]
because $s^{-1}=\frac{1}{n}s^T$. The ideal $I$ is generated
by all elements $b_i-\sum_{j=1}^n \lambda_{ij}s_j$ and with
respect to the basis $c_1,\ldots,c_{2^n}$, this becomes
\[ I=\langle
 \sum_m (T_{i,m}-\sum_{j=1}^n \lambda_{ij}T_{\tau(j),m}) c_m
 \mid i=1,\ldots,2^n \rangle .\]
Now it suffices to check that the coefficient
$T_{i,m}-\sum_{j=1}^n \lambda_{ij}T_{\tau(j),m}$ is zero exactly
for $m$ such that there exists $\mu$ with
$\sigma(b_m)_\nu=\delta_{\mu,\nu}$ for all $\nu$.
But for such an $m$,
\[ T_{i,m}=(-1)^{\sigma(b_i)_\mu}=b_{i,\mu} \]
and also $T_{\tau(j),m}=b_{\tau(j),\mu}$.
The assertion now follows from the orthogonality of $s$.
\end{proof}

\subsection*{Acknowledgment}

Section \ref{zbrwo} of this article generalizes and extends some
results of \cite{mC} which has been achieved under the supervision
of G. Malle. The author is also thankful to G. Malle for many
other valuable comments.

\iffalse
\nocite{CGR}
\nocite{Lus1}
\nocite{mC}
\nocite{mC1}
\nocite{bBaK}
\nocite{mGgM}
\nocite{gMSp}
\nocite{MUG}
\nocite{jN}
\nocite{wdW}
\nocite{hiBphZ}
\nocite{hB}
\nocite{zAeFmM}
\fi

\bibliographystyle{amsplain}
\bibliography{references}

\end{document}